\documentclass[smallextended]{svjour3}
\usepackage[T1]{fontenc}      % Encoding
\usepackage[utf8]{inputenc}   % Encoding
\DeclareUnicodeCharacter{00A0}{ } % to prevent \u8 errors
\usepackage{microtype}        % Improved typesetting
\usepackage{bbm}              % For getting \mathbbm{1} since \mathbb{1} does not work properly
\usepackage{amsmath}
\usepackage{amssymb}
\usepackage{commath}          % Commands for differentials etc.
\usepackage{mathtools}
\usepackage[numbers,square]{natbib}
\usepackage{verbatim}               % For multiline comments
\usepackage[dvipsnames]{xcolor}     % For writer-specific comments, must be loaded before tikz

\smartqed

\numberwithin{equation}{section}

\spnewtheorem{theorem}{Theorem}[section]{\bfseries}{\itshape}
\numberwithin{theorem}{section}
\spnewtheorem{lemma}[theorem]{Lemma}{\bfseries}{\itshape}
\spnewtheorem{proposition}[theorem]{Proposition}{\bfseries}{\itshape}
\spnewtheorem{definition}[theorem]{Definition}{\bfseries}{\itshape}
\spnewtheorem{corollary}[theorem]{Corollary}{\bfseries}{\itshape}
\spnewtheorem{remark}[theorem]{Remark}{\bfseries}{\rmfamily}

\DeclareMathOperator{\e}{e} % Exponential constant
\renewcommand{\b}[1]{\pmb{#1}} % Bolding

\DeclarePairedDelimiterX\Set[2]{\lbrace}{\rbrace}%
{ #1 \,:\, #2 }                                         % Set of type {A : B}
\DeclarePairedDelimiterX\inprod[2]{\langle}{\rangle}%
{ #1 , #2 }                                             % Inner product
\DeclareMathOperator*{\argmin}{arg\,min}  % Argmin
\DeclareMathOperator*{\argmax}{arg\,max}  % Argmax

\newcommand{\T}{\mathsf{T}} % Transpose

\newcommand{\R}{\mathbf{R}} % Set of real numbers
\newcommand{\N}{\mathbf{N}} % Set of natural numbers
\title{Kernel-based interpolation at approximate Fekete points}

\author{Toni Karvonen \and Simo Särkkä \and Ken'ichiro Tanaka}
\institute{T. Karvonen
\at Department of Electrical Engineering and Automation, Aalto University, Espoo, Finland \& The Alan Turing Institute, London, United Kingdom \\ \email{tkarvonen@turing.ac.uk}
\and S. Särkkä 
\at Department of Electrical Engineering and Automation, Aalto University, Espoo, Finland \\ \email{simo.sarkka@aalto.fi}
\and
K. Tanaka \at
Department of Mathematical Informatics, Graduate School of Information Science and Technology, University of Tokyo, Japan \\
\email{kenichiro@mist.i.u-tokyo.ac.jp}
}

\date{}

\begin{document}

\maketitle

\begin{abstract}
We construct approximate Fekete point sets for kernel-based interpolation by maximising the determinant of a kernel Gram matrix obtained via truncation of an orthonormal expansion of the kernel.
Uniform error estimates are proved for kernel interpolants at the resulting points.
If the kernel is Gaussian we show that the approximate Fekete points in one dimension are the solution to a convex optimisation problem and that the interpolants converge with a super-exponential rate.
Numerical examples are provided for the Gaussian kernel.
\end{abstract}

\section{Introduction}

Kernel-based methods are widely used in interpolation and approximation of functions~\citep{Wendland2005,Fasshauer2007,FasshauerMcCourt2015}.
Let $d \in \N$ and $\Omega \subset \R^d$ be a compact set with a non-empty interior. 
Given evaluations of a function $f \colon \Omega \to \R$ at a scattered set of distinct points $\mathcal{X}_n = \{ \b{x}_1, \ldots, \b{x}_n \} \subset \Omega$ and a continuous positive-definite kernel $K \colon \Omega \times \Omega \to \R$, the \emph{kernel interpolant} $s_{f}$ is 
\begin{equation*}
  s_f(\b{x}) = \sum_{k=1}^n c_k K(\b{x}, \b{x}_k),
\end{equation*}
where the coefficients $c_k$ are uniquely determined by the interpolation conditions \sloppy{${s_f(\b{x}_k) = f(\b{x}_k)}$} for every $k = 1,\ldots,n$.
The choice of the evaluation points $\mathcal{X}_n$ can have a significant effect on the accuracy of the approximation $s_f(\b{x}) \approx f(\b{x})$ at $\b{x} \notin \mathcal{X}_n$.
Popular methods for constructing ``good'' point sets include different types of greedy algorithms~\citep{SchabackWendland2000,DeMarchiSchabackWendland2005,Muller2009,WirtzHaasdonk2013,SantinHaasdonk2017} that construct the next point $\b{x}_{n+1}$ by maximising the power function.
An alternative approach is to select $n$ points concurrently by maximising
\begin{equation*}
  \det \mathcal{K}_{\mathcal{X}_n} = \det (K(\b{x}_k, \b{x}_m))_{k,m=1}^n,
\end{equation*}
the determinant of the kernel Gram matrix, over all sets of $n$ points $\mathcal{X}_n \subset \Omega$.
The resulting points are called \emph{Fekete points} in an analogue to the classical Fekete points that maximise the Vandermonde determinant~\citep{BosDeMarchi2010,BrianiSommarivaVianello2012}.
The asymptotic distribution of these points for kernel-based interpolation in one dimension has been studied by \citet{BosMaier2002} and \citet{BosDeMarchi2011}.

Because maximisation of $\det \mathcal{K}_{\mathcal{X}_n}$ is typically intractable, in this article we study \emph{approximate Fekete points} that are obtained by maximising the determinant of the kernel matrix of a truncated version of the kernel.
Let $\{\varphi_\ell\}_{\ell=1}^\infty$ be an orthonormal basis of~$\mathcal{H}_K(\Omega)$, the \emph{reproducing kernel Hilbert space} (RKHS) of $K$.
Then the kernel can be written as
\begin{equation*}
  K(\b{x}, \b{y}) = \sum_{\ell=1}^\infty \varphi_\ell(\b{x}) \varphi_\ell(\b{y}).
\end{equation*}
The approximate Fekete points $\mathcal{X}_n^*$ are then defined as any set of $n$ points that maximise
\begin{equation} \label{eq:fekete-intro}
  \det \widehat{\mathcal{K}}_{\mathcal{X}_n} = \det \Bigg( \sum_{\ell=1}^n \varphi_\ell(\b{x}_k) \varphi_\ell(\b{x}_m) \Bigg)_{k,m=1}^n.
\end{equation}
This and related constructions have been recently suggested by \citet{Tanaka2019} and, in the context of numerical integration and sampling from determinantal point processes, by \citet{Belhadji2019} and \citet{Gautier2019}.
Our construction differs slightly from the prior work in that we do not require the basis functions $\{\varphi_\ell\}_{\ell=1}^\infty$ to arise from Mercer's theorem, which significantly simplifies analysis and construction of the points, at least when the kernel is Gaussian.
This article contains two main theoretical contributions:
\begin{itemize}
\item Let $f \in \mathcal{H}_K(\Omega)$. In Section~\ref{sec:error-estimates} we use a bound on the Lebesgue constant for interpolation using $\{\varphi_\ell\}_{\ell=1}^n$ to prove that 
  \begin{equation} \label{eq:error-intro}
    \sup_{ \b{x} \in \Omega} \abs[0]{ f(\b{x}) - s_f(\b{x}) } \leq 2\norm[0]{f}_{\mathcal{H}_K(\Omega)} (1 + n) \sup_{\b{x} \in \Omega} \bigg( \sum_{\ell=n+1}^\infty \varphi_\ell(\b{x})^2 \bigg)^{1/2}
  \end{equation}
  for kernel interpolation at any approximate Fekete points.
\item In Section~\ref{sec:gauss-kernel} we show that for a certain simple orthonormal expansion~\citep{Minh2010} of the univariate Gaussian kernel
  \begin{equation*}
    K(x,y) = \exp\big( \! -\varepsilon^2 (x-y)^2 \big)
  \end{equation*}
  with a scale parameter $\varepsilon > 0$ the objective function~\eqref{eq:fekete-intro} is convex and has a unique maximiser.
  This is made possible by a convenient factorisation of the determinant in~\eqref{eq:fekete-intro} for this basis.
  We then specialise the uniform error estimate~\eqref{eq:error-intro} and some other results from Section~\ref{sec:error-estimates} for the Gaussian kernel.
\end{itemize}
Two numerical examples for the Gaussian kernel are given in Section~\ref{sec:examples}.
We also discuss improved error estimates in subspaces of $\mathcal{H}_K(\Omega)$ and tensor product extensions of the univariate approximate Fekete points for anisotropic multivariate Gaussian kernels.

\section{Background}

This section reviews basic properties of kernel interpolants and defines the approximate Fekete points studied in the remainder of the article.

\subsection{Kernel-based interpolation} \label{sec:interpolation}

Every positive-definite kernel $K \colon \Omega \times \Omega \to \R$ on a general domain \sloppy{${\Omega \subset \R^d}$} induces a unique reproducing kernel Hilbert space $\mathcal{H}_K(\Omega)$, which is a Hilbert space consisting of real-valued functions defined on $\Omega$. The RKHS is characterised by the properties that $K(\cdot, \b{x}) \in \mathcal{H}_K(\Omega)$ for every $\b{x} \in \Omega$ and $\inprod{f}{K(\cdot, \b{x})}_{\mathcal{H}_K(\Omega)} = f(\b{x})$ for every $f \in \mathcal{H}_K(\Omega)$ and $\b{x} \in \Omega$, the latter of which is known as the \emph{reproducing property}.

Given a set of $n$ distinct points, $\mathcal{X}_n = \{\b{x}_1, \ldots, \b{x}_n\} \subset \Omega$, the kernel interpolant $s_f$ is the minimum-norm interpolant to a function $f \colon \Omega \to \R$ at these points:
\begin{equation} \label{eq:minimum-norm}
  s_f = \argmin \Set[\big]{ \norm[0]{g}_{\mathcal{H}_K(\Omega)} }{ g \in \mathcal{H}_K(\Omega) \text{ s.t. } g(\b{x}_k) = f(\b{x}_k) \text{ for every } x_k \in \mathcal{X}_n }.
\end{equation}
This definition implies that $\norm[0]{s_f}_{\mathcal{H}_K(\Omega)} \leq \norm[0]{f}_{\mathcal{H}_K(\Omega)}$. The main advantage in working in an RKHS as opposed to some different function space is that the minimum-norm interpolant has a simple algebraic form:
\begin{equation} \label{eq:kernel-interpolant}
  s_f(\b{x}) = \sum_{k=1}^n c_k K(\b{x}, \b{x}_k) = \b{c}^\T \b{k}_{\mathcal{X}_n}(\b{x}),
\end{equation}
where we denote $\b{c} = (c_1, \ldots, c_n) \in \R^n$ and $\b{k}_{\mathcal{X}_n}(\b{x}) = (K(\b{x},\b{x}_k))_{k=1}^n \in \R^n$.
The coefficients $\b{c}$ are 
\begin{equation*}
  \begin{bmatrix} c_1 \\ \vdots \\ c_n \end{bmatrix} = \begin{bmatrix} K(\b{x}_1,\b{x}_1) & \cdots & K(\b{x}_1, \b{x}_n) \\ \vdots & \ddots & \vdots \\ K(\b{x}_n, \b{x}_1) & \cdots & K(\b{x}_n,\b{x}_n) \end{bmatrix}^{-1} \begin{bmatrix} f(\b{x}_1) \\ \vdots \\ f(\b{x}_n) \end{bmatrix},
\end{equation*}
where $\mathcal{K}_{\mathcal{X}_n} = (K(\b{x}_k,\b{x}_m))_{k,m=1}^n$ is the positive-definite kernel Gram matrix. 
From this it follows that $s_f$ is the unique interpolant to $f$ at $\mathcal{X}_n$ in the span of $\{K(\cdot,\b{x}_k)\}_{k=1}^n$. 

The interpolant can be written as $s_f = \sum_{k=1}^n f(\b{x}_k) u_k$ using the cardinal functions \sloppy{${u_k \in \operatorname{span} \{K(\cdot,\b{x}_k)\}_{k=1}^n}$} that satisfy $u_k(\b{x}_m) = \delta_{km}$.
From the reproducing property and the Cauchy--Schwarz inequality it then follows that for any $f \in \mathcal{H}_K(\Omega)$ the interpolation error admits the bound
\begin{equation} \label{eq:P-error-bound}
  \begin{split}
  \abs[0]{f(\b{x}) - s_f(\b{x})} &= \abs[4]{ \inprod[\bigg]{ f }{ K(\cdot, \b{x}) - \sum_{k=1}^n K(\cdot, \b{x}_k) u_k(\b{x}) }_{\mathcal{H}_K(\Omega)}} \\
  &\leq \norm[0]{f}_{\mathcal{H}_K(\Omega)} \norm[3]{ K(\cdot, \b{x}) - \sum_{k=1}^n K(\cdot, \b{x}_k) u_k(\b{x}) }_{\mathcal{H}_K(\Omega)} \\
  &\eqqcolon \norm[0]{f}_{\mathcal{H}_K(\Omega)} P_{\mathcal{X}_n}(\b{x}),
  \end{split}
\end{equation}
where the non-negative \emph{power function}, $P_{\mathcal{X}_n}$, can be alternatively expressed as
\begin{equation} \label{eq:power-function}
P_{\mathcal{X}_n}(\b{x}) = \sqrt{ K(\b{x}, \b{x}) - \b{k}_{\mathcal{X}_n}(\b{x})^\T \mathcal{K}_{\mathcal{X}_n}^{-1} \b{k}_{\mathcal{X}_n}(\b{x}) } = \sup_{ \norm[0]{f}_{\mathcal{H}_K(\Omega)} \leq 1} \abs[0]{ f(\b{x}) - s_f(\b{x}) }.
\end{equation}
The latter form is the point-wise worst-case approximation error.
The power function can be also written in a determinantal form~\citep[e.g.,][Lemma~3]{Schaback2005}
\begin{equation*}
  P_{\mathcal{X}_n}(\b{x}) = \frac{\det \mathcal{K}_{\mathcal{X}_n \cup \{\b{x}\}}}{\det \mathcal{K}_{\mathcal{X}_n}},
\end{equation*}
which suggests, via~\eqref{eq:P-error-bound}, that points $\mathcal{X}_n$ that maximise $\det \mathcal{K}_{\mathcal{X}_n}$ ought to provide small approximation error.
Numerous explicit bounds on the error $f-s_f$ in different norms and for different classes of kernels and functions within and without the RKHS can be found in \citep[Chapter~11]{Wendland2005} and \citep{WendlandRieger2005,NarcowichWardWendland2006,Arcangeli2007,RiegerZwicknagl2010}.

\subsection{Approximate Fekete points} \label{sec:points}

For the remainder of this article we assume that $\Omega$ is a compact subset of $\R^d$ with a non-empty interior and that the positive-definite kernel $K \colon \Omega \times \Omega \to \R$ is continuous.
These assumptions guarantee that the RKHS is separable~\citep[e.g.,][Proposition~11.7]{PaulsenRaghupathi2016}.
Let $\{\varphi_\ell\}_{\ell=1}^\infty$ be any orthonormal basis of $\mathcal{H}_K(\Omega)$.
Then the kernel can be written as 
\begin{equation} \label{eq:kernel-expansion}
  K(\b{x}, \b{y}) = \sum_{\ell=1}^\infty \varphi_\ell(\b{x}) \varphi_\ell(\b{y})
\end{equation}
for all $\b{x}, \b{y} \in \Omega$.
Note that there is an infinite number of different orthonormal bases of the RKHS and the expansion~\eqref{eq:kernel-expansion} is valid for each of them. For example, an infinitude of bases can be generated by varying the domain and measure in Mercer's theorem (see Section~\ref{sec:superconvergence}), though we do not assume that the basis $\{\varphi_\ell\}_{\ell=1}^\infty$ arises this way.
It is easy to verify that $K$ in~\eqref{eq:kernel-expansion} is the reproducing kernel: Any $f \in \mathcal{H}_K(\Omega)$ has the expansion $f = \sum_{\ell=1}^\infty \inprod{f}{\varphi_\ell}_{\mathcal{H}_K(\Omega)} \varphi_\ell$ so that
\begin{equation*}
  \begin{split}
  \inprod{f}{K(\cdot, \b{x})}_{\mathcal{H}_K(\Omega)} &= \sum_{\ell,k=1}^\infty \inprod{ \varphi_\ell}{ \varphi_k}_{\mathcal{H}_K(\Omega)} \inprod{f}{\varphi_\ell}_{\mathcal{H}_K(\Omega)} \varphi_k(\b{x}) \\
  &= \sum_{\ell=1}^\infty \inprod{f}{\varphi_\ell}_{\mathcal{H}_K(\Omega)} \varphi_\ell(\b{x}) \\
  &= f(\b{x}).
  \end{split}
\end{equation*}
The \emph{Fekete points} for interpolation with the kernel~\eqref{eq:kernel-expansion} are the points that maximise the determinant
\begin{equation} \label{eq:fekete-max}
  \det \mathcal{K}_{\mathcal{X}_n} = \det \begin{bmatrix} K(\b{x}_1,\b{x}_1) & \cdots & K(\b{x}_1, \b{x}_n) \\ \vdots & \ddots & \vdots \\ K(\b{x}_n,\b{x}_1) & \cdots & K(\b{x}_n,\b{x}_n) \end{bmatrix}
\end{equation}
of the kernel matrix.
As exact computation of the Fekete points is typically challenging, we fix an orthonormal basis $\{\varphi_\ell\}_{\ell=1}^\infty$ of $\mathcal{H}_K(\Omega)$, truncate the expansion~\eqref{eq:kernel-expansion} after $n$ terms and consider maximisation of the resulting approximation of the objective function~\eqref{eq:fekete-max}.
Define the truncated kernel
\begin{equation} \label{eq:truncated-kernel}
  \widehat{K}(\b{x}, \b{y}) = \sum_{\ell=1}^{n} \varphi_\ell(\b{x}) \varphi_\ell(\b{y})
\end{equation}
and its kernel matrix $\widehat{\mathcal{K}}_{\mathcal{X}_n} = (\widehat{K}(\b{x}_k,\b{x}_m))_{k,m=1}^n \in \R^{n \times n}$.
From~\eqref{eq:truncated-kernel} it is easy to see that
\begin{equation*}
  \widehat{\mathcal{K}}_{\mathcal{X}_n} = \Phi_{\mathcal{X}_n} \Phi_{\mathcal{X}_n}^\T, \quad \text{ where } \quad \Phi_{\mathcal{X}_n} = \begin{bmatrix} \varphi_1(\b{x}_1) & \cdots & \varphi_{n}(\b{x}_1) \\ \vdots & \ddots & \vdots \\ \varphi_1(\b{x}_n) & \cdots & \varphi_{n}(\b{x}_n) \end{bmatrix}.
\end{equation*}
The \emph{approximate Fekete points} $\mathcal{X}_n^* = \{\b{x}_1^*, \ldots, \b{x}_n^*\} \subset \Omega$ are then any points such that
\begin{equation} \label{eq:approximate-fekete}
  \mathcal{X}_n^* = \{\b{x}_1^*, \ldots, \b{x}_n^*\} \in \argmax_{ \mathcal{X}_n = \{\b{x}_1, \ldots, \b{x}_n\} \subset \Omega} \det \widehat{\mathcal{K}}_{\mathcal{X}_n} = \argmax_{ \mathcal{X}_n = \{\b{x}_1, \ldots, \b{x}_n\} \subset \Omega} \det \Phi_{\mathcal{X}_n}.
\end{equation}
Note that because $\{\varphi_\ell\}_{\ell=1}^n$ are linearly independent, there exists $\mathcal{X}_n \subset \Omega$ such that $\det \Phi_{\mathcal{X}_n} > 0$.
As $\Omega$ is compact and the continuity of $K$ implies the continuity of the basis functions, there exist points $\mathcal{X}_n^*$ at which $\det \Phi_{\mathcal{X}_n}$ attains a maximal value.

Given a set $\mathcal{X}_n$ of $n$ previously selected points, the popular $P$-greedy algorithm~\citep{DeMarchiSchabackWendland2005,SantinHaasdonk2017} selects $\b{x}_{n+1}$ such that
\begin{equation} \label{eq:p-greedy}
  \b{x}_{n+1} \in \argmax_{ \b{x} \in \Omega} \, P_{\mathcal{X}_n}(\b{x}),
\end{equation}
which, using the block determinant identity and~\eqref{eq:power-function}, can be written in the equivalent form
\begin{equation*}
  \b{x}_{n+1} \in \argmax_{ \b{x} \in \Omega} \, \det \begin{bmatrix} \mathcal{K}_{\mathcal{X}_n} & \b{k}_{\mathcal{X}_n}(\b{x}) \\ \b{k}_{\mathcal{X}_n}(\b{x})^\T & K(\b{x}, \b{x}) \end{bmatrix} = \argmax_{ \b{x} \in \Omega} \, \det \mathcal{K}_{\mathcal{X}_n \cup \{\b{x}\}}.
\end{equation*}
That is, the $P$-greedy points can be interpreted as greedily computed Fekete points.
Because it is known~\citep{SantinHaasdonk2017} that the interpolation error of the $P$-greedy algorithm decays fast (in some cases with an optimal rate), it is reasonable to expect that these rates are inherited or surpassed by interpolation at the Fekete points, and by extension perhaps by interpolation at the approximate Fekete points.
This is confirmed by numerical examples for the Gaussian kernel in Section~\ref{sec:examples}.

\section{Error estimates} \label{sec:error-estimates}

This section provides upper bounds on the error of approximating $f \in \mathcal{H}_K(\Omega)$ with the kernel interpolant $s_f$ when the interpolation points are the approximate Fekete points from Section~\ref{sec:points}.

\subsection{Interpolation with basis functions and Lebesgue constants}

For any $f \colon \Omega \to \R$ and any points $\mathcal{X}_n = \{\b{x}_1, \ldots, \b{x}_n\} \subset \Omega$ such that the matrix $\Phi_{\mathcal{X}_n} = (\varphi_m(\b{x}_k))_{k,m=1}^n$ is invertible there exists a unique interpolant $s_f^\varphi$ such that
\begin{itemize}
\item[(i)] $s_f^\varphi(\b{x}_k) = f(\b{x}_k)$ for every $k = 1 \ldots, n$;
\item[(ii)] $s_f^\varphi \in \operatorname{span}\{\varphi_\ell\}_{\ell=1}^{n}$.
\end{itemize}
From these requirements it follows that
\begin{equation} \label{eq:auxiliary-interpolant}
  s_f^\varphi = \sum_{k=1}^n c_k \varphi_{k},
\end{equation}
where the coefficients are
\begin{equation*}
  \begin{bmatrix} c_1 \\ \vdots \\ c_n \end{bmatrix} = \begin{bmatrix} \varphi_1(\b{x}_1) & \cdots & \varphi_{n}(\b{x}_1) \\ \vdots & \ddots & \vdots \\ \varphi_1(\b{x}_n) & \cdots & \varphi_{n}(\b{x}_n) \end{bmatrix}^{-1} \begin{bmatrix} f(\b{x}_1) \\ \vdots \\ f(\b{x}_n) \end{bmatrix}.
\end{equation*}
Alternatively, the interpolant can be written in the \emph{Lagrange form}
\begin{equation} \label{eq:lagrange-form}
  s_f^\varphi = \sum_{k=1}^n f(\b{x}_k) u_k^\varphi,
\end{equation}
where $u_k^\varphi$ are the \emph{Lagrange basis functions} solved from
\begin{equation} \label{eq:lagrange-system}
  \begin{bmatrix} \varphi_1(\b{x}_1) & \cdots & \varphi_{1}(\b{x}_{n}) \\ \vdots & \ddots & \vdots \\ \varphi_{n}(\b{x}_1) & \cdots & \varphi_{n}(\b{x}_n) \end{bmatrix} \begin{bmatrix} u_1^\varphi(\b{x}) \\ \vdots \\ u_k^\varphi(\b{x}) \end{bmatrix} = \begin{bmatrix} \varphi_1(\b{x}) \\ \vdots \\ \varphi_{n}(\b{x}) \end{bmatrix}
\end{equation}
for every $\b{x} \in \Omega$.
The \emph{Lebesgue constant} is defined using the Lagrange function as follows:
\begin{equation} \label{eq:lebesgue-constant}
  \Lambda_\varphi(\mathcal{X}_n) = \sup_{ \b{x} \in \Omega} \, \sum_{k=1}^n \abs[0]{u_k^\varphi(\b{x})}.
\end{equation}
A standard argument yields a conservative upper bound on the Lebesgue constant at approximate Fekete points~\citep{BosDeMarchi2010}.

\begin{proposition} \label{eq:lebesgue-fekete} If $\mathcal{X}_n^*$ are any approximate Fekete points~\eqref{eq:approximate-fekete}, then the Lebesgue constant~\eqref{eq:lebesgue-constant} satisfies
  \begin{equation} \label{eq:lebesgue-bound}
    \Lambda_\varphi(\mathcal{X}_n^*) \leq n.
  \end{equation}
\end{proposition}
\begin{proof}
Cramer's rule applied to~\eqref{eq:lagrange-system} gives
\begin{equation} \label{eq:cramer}
  u_k^\varphi(\b{x}) = \frac{\det \Phi^{k}_{\mathcal{X}_n}(\b{x})}{\det \Phi_{\mathcal{X}_n}},
\end{equation}
where $\Phi_{\mathcal{X}_n}^k(\b{x})$ is obtained by replacing the $k$th row of the matrix $\Phi_{\mathcal{X}_n}$ with the row vector $(\varphi_1(\b{x}), \ldots, \varphi_n(\b{x})) \in \R^n$.
  Because any approximate Fekete points maximise $\det \Phi_{\mathcal{X}_n}$ among all sets of $n$ points within $\Omega$ and $\Phi_{\mathcal{X}^*_n}^k(\b{x}) = \Phi_{\mathcal{X}_{n,k}^*(\b{x})}$ with \sloppy{${\mathcal{X}_{n,k}^*(\b{x}) = \{ \b{x}_1, \ldots, \b{x}_{k-1}, \b{x}, \b{x}_{k+1}, \ldots, \b{x}_n\}}$},
  \begin{equation*}
    \det \Phi_{\mathcal{X}^*_n} \geq \det \Phi_{\mathcal{X}^*_{n,k}(\b{x})} = \det \Phi_{\mathcal{X}^*_n}^k(\b{x}).
  \end{equation*}
  From~\eqref{eq:cramer} we thus get
  \begin{equation*}
    \Lambda_\varphi(\mathcal{X}_n^*) = \sup_{ \b{x} \in \Omega} \, \sum_{k=1}^n \abs[0]{u_k^\varphi(\b{x})} = \sup_{ \b{x} \in \Omega} \, \sum_{k=1}^n \abs[3]{ \frac{\det \Phi^k_{\mathcal{X}^*_n}(\b{x})}{\det \Phi_{\mathcal{X}^*_n}} } \leq \sup_{ \b{x} \in \Omega } \, \sum_{k=1}^n 1 = n.
  \end{equation*} \qed
\end{proof}

See~\citep{DeMarchiSchaback2010} for bounds on the Lebesgue constant for kernel interpolation, $\sup_{ \b{x} \in \Omega} \sum_{k=1}^n \abs[0]{u_k(\b{x})}$, when the RKHS is a Sobolev space.

\subsection{Uniform error estimates} \label{sec:error-uniform}

In this section we derive an estimate of the uniform interpolation error when $f$ is in the RKHS of $K$.
Recall that since $\{\varphi_\ell\}_{\ell=1}^\infty$ is an orthonormal basis of $\mathcal{H}_K(\Omega)$, any $f \in \mathcal{H}_K(\Omega)$ can be written as
\begin{equation} \label{eq:f-expansion}
  f = \sum_{\ell=1}^n f_\ell \varphi_\ell
\end{equation}
for a square-summable sequence of real coefficients $f_\ell = \inprod{f}{\varphi_\ell}_{\mathcal{H}_K(\Omega)}$.
The RKHS norm of $f$ in~\eqref{eq:f-expansion} is
\begin{equation} \label{eq:f-norm}
  \norm[0]{f}_{\mathcal{H}_K(\Omega)}^2 = \sum_{\ell=1}^\infty f_\ell^2.
\end{equation}
That is, $\mathcal{H}_K(\Omega)$ consists of functions having the form~\eqref{eq:f-expansion} such that their norm in~\eqref{eq:f-norm} is finite.
The following standard result on orthonormal expansions will be useful. Its proof consists of a straightforward application of the Cauchy--Schwarz inequality.

\begin{lemma} \label{lemma:g-lemma} If $f = \sum_{\ell=1}^\infty f_\ell \varphi_\ell \in \mathcal{H}_K(\Omega)$, then
  \begin{equation*}
    \abs[4]{ f(\b{x}) - \sum_{\ell=1}^n f_\ell \varphi_\ell(\b{x}) } \leq \norm[0]{f}_{\mathcal{H}_K(\Omega)} \bigg( \sum_{\ell=n+1}^\infty \varphi_\ell(\b{x})^2 \bigg)^{1/2}
  \end{equation*}
for every $\b{x} \in \Omega$.
\end{lemma}

\begin{theorem} \label{thm:uniform-error} Let $\mathcal{X}_n = \{ \b{x}_1, \ldots, \b{x}_n\} \subset \Omega$ be any points such that $\Phi_{\mathcal{X}_n}$ is invertible. Then for any $f \in \mathcal{H}_K(\Omega)$,
  \begin{equation} \label{eq:generic-bound}
    \sup_{ \b{x} \in \Omega} \, \abs[0]{ f(\b{x}) - s_f(\b{x})} \leq 2\norm[0]{f}_{\mathcal{H}_K(\Omega)} (1 + \Lambda_\varphi(\mathcal{X}_n)) \sup_{\b{x} \in \Omega} \bigg( \sum_{\ell=n+1}^\infty \varphi_\ell(\b{x})^2 \bigg)^{1/2}.
  \end{equation}
\end{theorem}
\begin{proof}
  Let $f = \sum_{\ell=1}^\infty f_\ell \varphi_\ell \in \mathcal{H}_K(\Omega)$ and define \sloppy{${g = \sum_{\ell=1}^n f_\ell \varphi_\ell}$}.
  Then
  \begin{equation*}
    \abs[0]{ f(\b{x}) - s_f^\varphi(\b{x}) } \leq \abs[0]{ f(\b{x}) - g(\b{x})} + \abs[0]{ g(\b{x}) - s_g^\varphi(\b{x}) } + \abs[0]{ s_g^\varphi(\b{x}) - s_f^\varphi(\b{x}) }.
  \end{equation*}
  The first term on the right-hand side can be bounded with Lemma~\ref{lemma:g-lemma}.
  The second term vanishes because $g \in \operatorname{span} \{ \varphi_\ell\}_{\ell=1}^n$ and $s_g^\varphi$ being the unique interpolant to $g$ in $\operatorname{span} \{ \varphi_\ell\}_{\ell=1}^n$ imply that $s_g^\varphi = g$.
  Finally, the Lagrange form~\eqref{eq:lagrange-form} and Lemma~\ref{lemma:g-lemma} yield a bound on the third term:
  \begin{equation*}
    \begin{split}
      \abs[0]{s_g^\varphi(\b{x}) - s_f^\varphi(\b{x})} &= \abs[4]{ \sum_{k=1}^n [ g(\b{x}_k) - f(\b{x}_k) ] u_k^\varphi(\b{x}) } \\
      &\leq \norm[0]{f}_{\mathcal{H}_K(\Omega)} \sum_{k=1}^n \abs[0]{u_k^\varphi(\b{x})} \bigg( \sum_{\ell=n+1}^\infty \varphi_\ell(\b{x}_k)^2 \bigg)^{1/2} \\
      &\leq \norm[0]{f}_{\mathcal{H}_K(\Omega)} \Lambda_\varphi(\mathcal{X}_n) \sup_{\b{x} \in \Omega} \bigg( \sum_{\ell=n+1}^\infty \varphi_\ell(\b{x})^2 \bigg)^{1/2}.
      \end{split}
  \end{equation*}
  Therefore,
  \begin{equation} \label{eq:sphi-interpolant-error}
    \abs[0]{ f(\b{x}) - s_f^\varphi(\b{x}) } \leq \norm[0]{f}_{\mathcal{H}_K(\Omega)} (1 + \Lambda_\varphi(\mathcal{X}_n)) \sup_{\b{x} \in \Omega} \bigg( \sum_{\ell=n+1}^\infty \varphi_\ell(\b{x})^2 \bigg)^{1/2}.
  \end{equation}
  To obtain a bound on $\abs[0]{f(\b{x}) - s_f(\b{x})}$ observe that
  \begin{equation*}
    \abs[0]{ f(\b{x}) - s_f(\b{x}) } \leq \abs[0]{ f(\b{x}) - s_f^\varphi(\b{x}) } + \abs[0]{ s_f^\varphi(\b{x}) - s_f(\b{x}) },
  \end{equation*}
  where, because $s_f^\varphi(\b{x}_k) = s_f(\b{x}_k) = f(\b{x}_k)$ for $k=1,\ldots,n$ and \sloppy{${\norm[0]{s_f}_{\mathcal{H}_K(\Omega)} \leq \norm[0]{f}_{\mathcal{H}_K(\Omega)}}$} by the norm-minimality property~\eqref{eq:minimum-norm}, both terms on the right-hand side obey the bound~\eqref{eq:sphi-interpolant-error}.
  The claim follows. \qed
\end{proof}

Proposition~\ref{eq:lebesgue-fekete} immediately yields an error estimate for any approximate Fekete points.

\begin{corollary} \label{corollary:error-fekete} Suppose that $f \in \mathcal{H}_K(\Omega)$ is interpolated at any approximate Fekete points~\eqref{eq:approximate-fekete}. Then
  \begin{equation} \label{eq:generic-bound-fekete}
    \sup_{ \b{x} \in \Omega} \, \abs[0]{ f(\b{x}) - s_f(\b{x})} \leq 2\norm[0]{f}_{\mathcal{H}_K(\Omega)} (1 + n) \sup_{\b{x} \in \Omega} \bigg( \sum_{\ell=n+1}^\infty \varphi_\ell(\b{x})^2 \bigg)^{1/2}.
  \end{equation}
\end{corollary}

Due to the presence of a supremum on the right-hand side of~\eqref{eq:generic-bound} and~\eqref{eq:generic-bound-fekete} it is difficult to make the bounds explicitly dependent on, for example, smoothness of the kernel as is usual in the error analysis of radial basis function interpolants~\citep[Chapter~11]{Wendland2005}.
One would ideally select a basis $\{\varphi_\ell\}_{\ell=1}^\infty$ that minimises the supremum in~\eqref{eq:generic-bound-fekete}. This seems challenging, so in practice selection of the basis is dictated by convenience, that is, by one's ability to derive an explicit bound for the supremum and the ease of implementation of the optimisation problem~\eqref{eq:approximate-fekete}.

\subsection{Improved error estimates in subspaces} \label{sec:superconvergence}

It is known that the rate of convergence of kernel interpolation can be improved if the function being interpolated lives in a subset of the RKHS. The existing results in \mbox{\citep{Schaback1999,Schaback2000,Schaback2018}} and \citep[Section~11.5]{Wendland2005} are particularly interesting when the kernel is finitely smooth\footnote{\citet[p.\ 192]{Wendland2005} goes as far as describing these results ``almost pointless'' for kernels, such as the Gaussian, that are associated with exponential rates of convergence.}. Roughly speaking, in this case a typical algebraic rate of convergence is ``doubled'' for sufficiently smooth elements of the RKHS.
Specifically, let $\mu$ be a Borel measure on $\Omega$ that assigns positive measure to every open set and let $\{\psi_\ell\}_{\ell=1}^\infty$ and $(\lambda_\ell)_{\ell=1}^\infty$ be the eigenfunctions and the positive decreasing eigenvalues of the integral operator $Tf(\b{x}) = \int_\Omega K(\b{x},\b{y}) f(\b{y}) \dif \mu(\b{y})$. By Mercer's theorem~\citep[e.g.,][]{Sun2005},
\begin{equation*}
  \mathcal{H}_K(\Omega) = \Set[\Bigg]{ f \in L^2(\mu) }{ \norm[0]{f}_{\mathcal{H}_K(\Omega)}^2 = \sum_{\ell=1}^\infty \frac{ \inprod{f}{\psi_\ell}_{L^2(\mu)}^2}{\lambda_\ell} < \infty }.
\end{equation*}
The standard improved error estimate states that for $f \in \mathcal{H}_K(\Omega)$ such that $f = Tv$ for some $v \in L^2(\mu)$ the bound~\eqref{eq:P-error-bound} is improved to
\begin{equation} \label{eq:improved-bound}
  \abs[0]{f(\b{x}) - s_f(\b{x})} \leq \norm[0]{v}_{L^2(\mu)} P_{\mathcal{X}_n}(\b{x}) \norm[0]{P_{\mathcal{X}_n}}_{L^2(\mu)}.
\end{equation}
Because the range of $T$ is
\begin{equation*}
  T(L^2(\mu)) = \Set[\Bigg]{ f \in L^2(\mu) }{ \norm[0]{f}_{\mathcal{H}_K(\Omega)}^2 = \sum_{\ell=1}^\infty \frac{ \inprod{f}{\psi_\ell}_{L^2(\mu)}^2}{\lambda_\ell^2} < \infty } \subset \mathcal{H}_K(\Omega),
\end{equation*}
the collection of functions for which~\eqref{eq:improved-bound} holds is a subset of the RKHS.
Theorem~\ref{thm:superconvergence} below is significantly more flexible than this result and does not require that the Mercer expansion be used.

Let $(\alpha_\ell)_{\ell=1}^\infty$ be a positive, increasing, and divergent sequence and define the subspace
\begin{equation*}
  \mathcal{H}_K^\alpha(\Omega) = \Set[\Bigg]{ f = \sum_{\ell=1}^\infty f_\ell \varphi_\ell }{ \norm[0]{f}_{\mathcal{H}_K^\alpha(\Omega)}^2 = \sum_{\ell = 1}^\infty \alpha_\ell^2 f_\ell^2 < \infty } \subset \mathcal{H}_K(\Omega).
\end{equation*}
For simplicity we also assume that $\alpha_1 \geq 1$, which can always be achieved using a scaling that does not affect $\mathcal{H}_K^\alpha(\Omega)$ as a set.
It is easy to verify that $\mathcal{H}_K^\alpha(\Omega)$ is an RKHS and that its reproducing kernel is
\begin{equation*}
  K^\alpha(\b{x},\b{y}) = \sum_{\ell=1}^\infty \frac{1}{\alpha_\ell^2} \, \varphi_\ell(\b{x}) \varphi_\ell(\b{y}).
\end{equation*}

\begin{theorem} \label{thm:superconvergence} Suppose that $f \in \mathcal{H}_K^\alpha(\Omega)$ is interpolated at any approximate Fekete points~\eqref{eq:approximate-fekete}. Then
  \begin{equation*}
    \sup_{\b{x} \in \Omega} \, \abs[0]{ f(\b{x}) - s_f(\b{x})} \leq 2\norm[0]{f}_{\mathcal{H}_K^\alpha(\Omega)} (1 + n) \alpha_{n+1}^{-1} \sup_{\b{x} \in \Omega} \bigg( \sum_{\ell=n+1}^\infty \varphi_\ell(\b{x})^2 \bigg)^{1/2}.
  \end{equation*}
\end{theorem}
\begin{proof}
  When $f \in \mathcal{H}_K^\alpha(\Omega)$, we replace the estimate of Lemma~\ref{lemma:g-lemma} with the following estimate:
  \begin{equation*}
  \begin{split}
    \abs[4]{ f(\b{x}) - \sum_{\ell=1}^{n} f_\ell \varphi_\ell(\b{x}) }^2 &= \abs[4]{ \sum_{\ell=n+1}^\infty \alpha_\ell f_\ell \alpha_\ell^{-1} \varphi_\ell(\b{x}) }^2 \\
    &\leq \bigg( \sum_{\ell=n+1}^\infty \alpha_\ell^2 f_\ell^2 \bigg) \bigg( \sum_{\ell=n+1}^\infty \alpha_\ell^{-2} \varphi_\ell(\b{x})^2 \bigg) \\
    &\leq \norm[0]{f}_{\mathcal{H}_K^\alpha(\Omega)}^2 \alpha_{n+1}^{-2} \sum_{\ell = n+1}^\infty \varphi_\ell(\b{x})^2.
    \end{split}
  \end{equation*}
  The proof of Theorem~\ref{thm:uniform-error} and the fact that $\norm[0]{f}_{\mathcal{H}_K(\Omega)} \leq \norm[0]{f}_{\mathcal{H}_K^\alpha(\Omega)}$, which follows from our assumption $\alpha_\ell \geq 1$ for every $\ell$, then yield the claimed uniform bound. \qed
\end{proof}

\section{Gaussian kernel} \label{sec:gauss-kernel}

The $d$-dimensional anisotropic Gaussian kernel
\begin{equation} \label{eq:gauss-multi}
  K(\b{x}, \b{y}) = \exp \bigg( \! - \sum_{i=1}^d \varepsilon_i^2 (x_i - y_i)^2 \bigg)
\end{equation}
with scale parameters $\varepsilon_i > 0$ has the orthonormal expansion
\begin{equation*}
  \begin{split}
    K(\b{x}, \b{y}) &= \sum_{ \b{\alpha} \in \N_0^d } \Bigg( \sqrt{ \frac{2^{\abs[0]{\b{\alpha}}} \b{\varepsilon}^{\b{\alpha}}}{\b{\alpha}!}} \b{x}^{\b{\alpha}} \exp\bigg(\!-\sum_{i=1}^d \varepsilon_i^2 x_i^2 \bigg) \Bigg) \\
    &\hspace{2cm}\times\Bigg( \sqrt{ \frac{2^{\abs[0]{\b{\alpha}}} \b{\varepsilon}^{\b{\alpha}}}{\b{\alpha}!}} \b{y}^{\b{\alpha}} \exp\bigg(\!-\sum_{i=1}^d \varepsilon_i^2 y_i^2 \bigg) \Bigg) \\
    &\eqqcolon \sum_{ \b{\alpha} \in \N_0^d } \varphi_{\b{\alpha}}(\b{x}) \varphi_{\b{\alpha}}(\b{y}),
    \end{split}
\end{equation*}
where $\N_0^d$ is the collection of $d$-dimensional non-negative multi-indices $\b{\alpha}$, $\abs[0]{\b{\alpha}} = \alpha_1 + \cdots + \alpha_d$, $\b{\alpha}! = \alpha_1! \times \cdots \times \alpha_d!$, and $\b{z}^{\b{\alpha}} = z_i^{\alpha_1} \times \cdots \times z_d^{\alpha_d}$ for any $\b{z} \in \R^d$.
This expansion can be verified via a straightforward calculation.
The RKHS of~\eqref{eq:gauss-multi} is thus
\begin{equation*}
  \mathcal{H}_K(\Omega) = \Set[\Bigg]{ f(\b{x}) = \sum_{ \b{\alpha} \in \N_0^d } f_{\b{\alpha}} \sqrt{ \frac{2^{\abs[0]{\b{\alpha}}} \b{\varepsilon}^{\b{\alpha}}}{\b{\alpha}!}} \b{x}^{\b{\alpha}} \exp\bigg(\!-\sum_{i=1}^d \varepsilon_i^2 x_i^2 \bigg) }{ \sum_{ \b{\alpha} \in \N_0^d } f_{\b{\alpha}}^2 < \infty}.
\end{equation*}
However, for the most of this section we set $d=1$ and consider the one-dimensional Gaussian kernel
\begin{equation} \label{eq:gauss-kernel}
  K(x,y) = \exp\big(\!-\varepsilon^2 (x-y)^2 \big)
\end{equation}
with a single scale parameter $\varepsilon > 0$.
The orthonormal expansion and the RKHS are then\footnote{Observe that in this section we begin indexing of the expansion from zero to simplify notation.}
\begin{equation} \label{eq:gauss-expansion}
  \begin{split}
  K(x, y) &= \sum_{\ell=0}^\infty \bigg( \sqrt{\frac{2^\ell \varepsilon^{2\ell}}{\ell!}} x^\ell \exp(-\varepsilon^2 x^2) \bigg) \bigg( \sqrt{\frac{2^\ell \varepsilon^{2\ell}}{\ell!}} y^\ell \exp(-\varepsilon^2 y^2) \bigg) \\
  &\eqqcolon \sum_{\ell=0}^\infty \varphi_\ell(x) \varphi_\ell(y)
  \end{split}
\end{equation}
and
\begin{equation} \label{eq:gauss-rkhs}
  \mathcal{H}_K(\Omega) = \Set[\Bigg]{ f(x) = \sum_{\ell=0}^\infty f_\ell \sqrt{\frac{2^\ell \varepsilon^{2\ell}}{\ell!}} x^\ell \exp(-\varepsilon^2 x^2) }{ \sum_{\ell=0}^\infty f_\ell^2 < \infty}.
\end{equation}
The above results and other properties of the Gaussian kernel and its RKHS are studied in more detail in \citep{Steinwart2006,Minh2010} and~\citep[Section~4]{DeMarchiSchaback2009}.
In Section~\ref{sec:fekete-gaussian} we show that, owing to the special structure of the above basis functions and the resulting convenient factorisation of $\det \Phi_{\mathcal{X}_n}$, the approximate Fekete points for the one-dimensional Gaussian kernel are solved from a convex optimisation problem.
Note that most prior work, such as \citep{Tanaka2019,Belhadji2019}, uses a well-known Mercer expansion of the Gaussian kernel instead of~\eqref{eq:gauss-expansion}.
This expansion is
\begin{equation} \label{eq:gauss-mercer}
  K(x,y) = \sum_{\ell=0}^\infty \lambda_\ell^\sigma \psi_\ell^\sigma(x) \psi_\ell^\sigma(y),
\end{equation}
where the eigenfunctions are orthonormal with respect to the Gaussian measure with variance $\sigma^2$:
\begin{equation*}
  \frac{1}{\sqrt{2 \pi \sigma^2}} \int_\R \psi_\ell^\sigma(x) \psi_k^\sigma(x) \exp\bigg(\!-\frac{x^2}{2\sigma^2} \bigg) \dif x = \delta_{\ell k}.
\end{equation*}
The eigenfunctions and values are~\citep{FasshauerMcCourt2012}
\begin{equation*}
  \psi_\ell^\sigma(x) = \sqrt{\frac{\beta}{\ell!}} \e^{-\delta^2 x^2} \mathrm{H}_\ell\big( \sqrt{2}\alpha \beta x \big) \: \text{ and } \: \lambda_\ell^\sigma = \sqrt{\frac{\alpha^2}{\alpha^2 + \delta^2 + \varepsilon^2}} \bigg( \frac{\varepsilon^2}{\alpha^2 + \delta^2 + \varepsilon^2} \bigg)^\ell,
\end{equation*}
where $\mathrm{H}_\ell$ is the $\ell$th probabilists' Hermite polynomial and the constants are
\begin{equation*}
  \alpha = \frac{1}{\sqrt{2} \sigma}, \quad \beta = (1 + 8\varepsilon^2 \sigma^2 )^{1/4} \quad \text{ and } \quad \delta^2 = \frac{1}{4\sigma^2}( \beta^2 - 1).
\end{equation*}
The Mercer expansion~\eqref{eq:gauss-mercer} can be then verified by inserting
\begin{equation*}
  \rho = \frac{\varepsilon^2}{\alpha^2+\delta^2+\varepsilon^2} \quad \text{ and } \quad \gamma = \sqrt{2} \alpha \beta 
\end{equation*}
into the Mehler formula
\begin{equation*}
  \exp\bigg( \!-\frac{\rho^2\gamma^2(x^2+y^2)-2\rho\gamma^2 xy}{2(1-\rho^2)}\bigg) = \sqrt{1-\rho^2} \sum_{\ell=0}^\infty \frac{\rho^\ell}{\ell!} \, \mathrm{H}_\ell(\gamma x) \, \mathrm{H}_\ell(\gamma y),
\end{equation*}
and multiplying both sides with
\begin{equation*}
   \exp\bigg(\!-\frac{\rho\gamma^2}{2(1+\rho)} (x^2+y^2) \bigg)= \exp\big(\!-\delta^2(x^2+y^2)\big).
\end{equation*}
The expansion~\eqref{eq:gauss-expansion} used in this article is evidently much simpler to work with.

\subsection{Approximate Fekete points via convex optimisation} \label{sec:fekete-gaussian}

Let
\begin{equation*}
  \begin{split}
  \widehat{K}(x,y) &= \sum_{\ell=0}^{n-1} \bigg( \sqrt{\frac{2^\ell \varepsilon^{2\ell}}{\ell!}} x^\ell \exp(-\varepsilon^2 x^2) \bigg) \bigg( \sqrt{\frac{2^\ell \varepsilon^{2\ell}}{\ell!}} y^\ell \exp(-\varepsilon^2 y^2) \bigg) \\
  &= \sum_{\ell=0}^{n-1} \varphi_\ell(x) \varphi_\ell(y)
  \end{split}
\end{equation*}
be the truncation of the Gaussian kernel~\eqref{eq:gauss-kernel} and $\widehat{\mathcal{K}}_{\mathcal{X}_n} = (\widehat{K}(x_k,x_m))_{k,m=1}^n \in \R^{n \times n}$ the corresponding kernel matrix. Define the matrices
\begin{equation*}
  \Phi_{\mathcal{X}_n} = \begin{bmatrix} \varphi_0(x_1) & \cdots & \varphi_{n-1}(x_1) \\ \vdots & \ddots & \vdots \\ \varphi_0(x_n) & \cdots & \varphi_{n-1}(x_n) \end{bmatrix} \quad \text{ and } \quad \mathcal{V}_{\mathcal{X}_n} = \begin{bmatrix} 1 & x_1 & \cdots & x_1^{n-1} \\ \vdots & \vdots & \ddots & \vdots \\ 1 & x_n & \cdots & x_n^{n-1} \end{bmatrix},
\end{equation*}
the latter of which is the classical Vandermonde matrix.
Since $\widehat{\mathcal{K}}_{\mathcal{X}_n} = \Phi_{\mathcal{X}_n} \Phi_{\mathcal{X}_n}^\T$ and the $k$th row of the matrix $\Phi_{\mathcal{X}_n}$ is that of the matrix $\mathcal{V}_{\mathcal{X}_n}$ multiplied by $(2^{k-1} \varepsilon^{2(k-1)}/(k-1)!)^{1/2} \exp(-\varepsilon^2 x_k^2)$, we have
\begin{equation*}
  \begin{split}
    \big( \det \widehat{\mathcal{K}}_{\mathcal{X}_n} \big)^{1/2} &= \abs[0]{ \det \Phi_{\mathcal{X}_n} } \\
    &= \bigg( \prod_{\ell=0}^{n-1} \frac{2^\ell \varepsilon^{2\ell}}{\ell!} \bigg)^{1/2}  \exp\bigg( \! -\varepsilon^2 \sum_{k=1}^n x_k^2 \bigg) \abs[0]{ \det \mathcal{V}_{\mathcal{X}_n} } \\
  &= \bigg( \prod_{\ell=0}^{n-1} \frac{2^\ell \varepsilon^{2\ell}}{\ell!} \bigg)^{1/2} \exp\bigg( \! -\varepsilon^2 \sum_{k=1}^n x_k^2 \bigg) \abs[3]{ \prod_{1 \leq i < j \leq n} (x_i - x_j) },
  \end{split}
\end{equation*}
where the last equation uses the standard explicit expression for the Vandermonde determinant.
This expression verifies that $\widehat{\mathcal{K}}_{\mathcal{X}_n}$ and $\Phi_{\mathcal{X}_n}$ are invertible whenever the points are distinct.
Define
\begin{equation} \label{eq:W-function}
  W(x_1, \ldots, x_n) = \exp\bigg( \! -\varepsilon^2 \sum_{k=1}^n x_k^2 \bigg) \abs[3]{ \prod_{1 \leq i < j \leq n} (x_i - x_j) }.
\end{equation}
The approximate Fekete points~\eqref{eq:approximate-fekete} for the Gaussian kernel are thus seen to be
\begin{equation} \label{eq:approximate-fekete-gaussian}
  \begin{split}
  \mathcal{X}_n^* = \{ x_1^*, \ldots, x_n^*\} &\in \argmax_{ \mathcal{X}_n = \{ x_1, \ldots, x_n\} \subset \Omega} \det \widehat{\mathcal{K}}_X \\
  &= \argmax_{ \mathcal{X}_n = \{ x_1, \ldots, x_n\} \subset \Omega} W(x_1, \ldots, x_n).
  \end{split}
\end{equation}
Maximisation of $W(x_1,\ldots,x_n)$ is equivalent to minimisation of the energy
\begin{equation*}
  \begin{split}
  I(x_1, \ldots, x_n) = - \log W(x_1, \ldots, x_n) &= \varepsilon^2 \sum_{k=1}^n x_k^2 + \sum_{1 \leq i < j \leq n} \log \frac{1}{\abs[0]{x_i - x_j}} \\
  &= \sum_{k=1}^n Q_\varepsilon(x_k) + \sum_{1 \leq i < j \leq n} N(x_i-x_j),
  \end{split}
\end{equation*}
where $Q_\varepsilon(x) = \varepsilon^2 x$ and $N(x) = 1/\log\abs[0]{x}$.
To ensure that $I$ is well-defined and to eliminate non-uniqueness arising from ordering of the points, define the simplex
\begin{equation*}
  \mathcal{R}_n = \Set[\big]{(x_1, \ldots, x_n) \in \Omega^n}{ x_1 < x_2 < \cdots < x_{n-1} < x_n} \subset \Omega^n
\end{equation*}
and consider $I$ as a function defined on $\mathcal{R}_n$.
Adaptation of the proof of Theorem~3.3 of \citet{TanakaSugihara2018} shows that the objective function $I$ is convex and that there exists a unique minimiser \sloppy{${\mathcal{X}_n^* \in \mathcal{R}_n}$}.

\begin{proposition} \label{prop:convex} If $\Omega \subset \R$ is a closed interval, then the energy function $I \colon \mathcal{R}^n \to \R$ is convex and has a unique minimiser.
\end{proposition}
\begin{proof}
The Hessian matrix $\nabla^2 I$ of $I$ is
\begin{equation*}
(\nabla^2 I)_{ij} = \frac{\partial^{2} I}{\partial x_{i} \partial x_{j}}
=
\begin{cases}
Q_\varepsilon''(x_{i}) + \sum_{k \neq i} N''(x_{i} - x_{k})  & (i = j), \\
- N''(x_{i} - x_{j}) & (i\neq j).
\end{cases}
\end{equation*}
Because both
\begin{equation*}
  N(x) = \log \frac{1}{\abs[0]{x}} \quad \text{ and } \quad Q_\varepsilon(x) = \varepsilon^2 x^2
\end{equation*}
are strictly convex on $\R \setminus \{ 0 \}$ and $\R$, respectively, we have $N'' > 0$ and $Q'' > 0$.
Therefore the diagonal elements of $\nabla^2 I$ are always positive.
Moreover,
\begin{equation} \label{eq:diagonal-dominance}
\sum_{k \neq i} \abs[0]{-N''(x_{i} - x_{k})}
=
\sum_{k \neq i} N''(x_{i} - x_{k})
<
 \sum_{k \neq i} N''(x_{i} - x_{k})
+ Q_\varepsilon''(x_{i}),
\end{equation}
which verifies that the Hessian is diagonally dominant and hence positive-definite.
That is, the energy function $I$ is convex on $\mathcal{R}_n$.

To verify that there is a unique minimiser in the non-closed set $\mathcal{R}_n$, consider the function $J(\mathcal{X}_n) = \exp(-I(\mathcal{X}_n))$ which is continuous on the closure of $\mathcal{R}_n$ if we set $J(\mathcal{X}_n) = 0$ for every $\mathcal{X}_n = \{x_1, \ldots, x_n\} \in \Omega^n$ such that $x_i = x_{i+1}$ for some $i$. 
Being positive on $\mathcal{R}_n$, any maximiser of $J$ is in $\mathcal{R}_n$.
As a maximiser of $J$ is a minimser of $I$ and $I$ is convex it follows that $I$ must have a unique minimiser in $\mathcal{R}_n$. \qed
\end{proof}

\begin{remark} \label{rmk:flat-limit} If we set $\varepsilon = 0$, the above optimisation problem becomes that of finding the Fekete points for polynomial interpolation. However, in this case the objective function $I$ is no longer convex because $Q_\varepsilon''(x_i) = 0$ in~\eqref{eq:diagonal-dominance}. Our optimisation problem can be thus viewed as a regularised version of the standard Fekete problem. Based on this and the well-known convergence\footnote{In one dimension the convergence occurs for any points and most commonly used infinitely smooth radial kernels but in higher dimensions the Gaussian kernel is special in that it is the only known kernel for which convergence to a polynomial interpolant, of minimal degree in a certain sense, occurs for every point set.} of kernel interpolants to polynomial interpolants at the so-called flat limit~\citep{Schaback2005,LeeYoonYoon2007,KarvonenSarkka2019b} it may be expected that $\mathcal{X}_n^*$ converge to the polynomial Fekete points as $\varepsilon \to 0$. We do not attempt to prove this.
\end{remark}

\subsection{Error estimates} \label{sec:convergence-gaussian}

In this section we denote $c_\Omega = \sup_{x \in \Omega} \abs[0]{x} < \infty$.

\begin{lemma} \label{lemma:phi-bound-gaussian} Consider the basis functions~\eqref{eq:gauss-expansion} and assume that $n \geq 2\varepsilon^2 c_\Omega^2$. Then
  \begin{equation*}
    \sup_{x \in \Omega} \bigg( \sum_{\ell=n}^\infty \varphi_\ell(x)^2 \bigg)^{1/2} \leq \frac{ \big(\sqrt{2}\,\varepsilon c_\Omega \big)^n}{\sqrt{n!}}.
  \end{equation*}
\end{lemma}
\begin{proof}
  By differentation it is easy to see that $\varphi_\ell(x)^2$ attains its maximal value on $\R$ at \sloppy{${x_0 \pm (\ell/(2\varepsilon^2))^{1/2}}$} and that $\varphi_\ell^2$ is decreasing on $[-(\ell/(2\varepsilon^2))^{1/2}, 0]$ and increasing on $[0,(\ell/(2\varepsilon^2))^{1/2}]$.
It follows that
\begin{equation*}
  \sup_{ x \in \Omega} \varphi_\ell(x)^2 = \varphi_\ell(c_\Omega)^2 = \frac{2^\ell \varepsilon^{2\ell}}{\ell!} c_\Omega^{2\ell} \exp(-2\varepsilon^2 c_\Omega^2)
\end{equation*}
for every $\ell \geq n$ if  $n \geq 2\varepsilon^2 c_\Omega^2$. 
  By Taylor's theorem there is $\xi \in [0, 2 \varepsilon^{2} c_\Omega^2]$ such that
  \begin{equation*}
    \begin{split}
    \sup_{x \in \Omega} \, \sum_{\ell = n}^\infty \varphi_\ell(x)^2 &\leq \exp(-2\varepsilon^2 c_\Omega^2) \sum_{\ell = n}^\infty \frac{(2 \varepsilon^{2} c_\Omega^{2})^\ell}{\ell!} \\
    &= \exp(-2\varepsilon^2 c_\Omega^2) \frac{\exp(\xi)}{n!} (2 \varepsilon^{2} c_\Omega^{2})^n \\
    &\leq \exp(-2\varepsilon^2 c_\Omega^2) \frac{\exp(2\varepsilon^2 c_\Omega^2)}{n!} (2 \varepsilon^{2} c_\Omega^{2})^n \\
    &= \frac{(2 \varepsilon^{2} c_\Omega^{2})^n}{n!}.
    \end{split}
  \end{equation*}
  This proves the claim. \qed
\end{proof}

Using the estimate of Lemma~\ref{lemma:phi-bound-gaussian} in Corollary~\ref{corollary:error-fekete} yields an explicit error estimate for interpolation with the Gaussian kernel.

\begin{theorem} \label{thm:error-gaussian} Consider the Gaussian kernel with the orthonormal expansion~\eqref{eq:gauss-expansion} and suppose that $\Omega \subset \R$ is a closed interval. If $f \in \mathcal{H}_K(\Omega)$ is interpolated at the unique approximate Fekete points $\mathcal{X}_n^*$ defined in~\eqref{eq:approximate-fekete-gaussian} and $n \geq 2\varepsilon^2 c_\Omega^2$, then
  \begin{equation} \label{eq:error-bound-gaussian}
    \sup_{x \in \Omega} \, \abs[0]{ f(x) - s_f(x)} \leq C_1 \norm[0]{f}_{\mathcal{H}_K(\Omega)} n^{3/4} \exp\bigg( \! -n \bigg( \frac{1}{2} \log n - \log C_2 \bigg)\bigg),
  \end{equation}
  where $C_1 = (128/\pi)^{1/4} \approx 2.53$ and $C_2 = \sqrt{2\e}\, \varepsilon c_\Omega$.
\end{theorem}
\begin{proof} The claim follows from Corollary~\ref{corollary:error-fekete}, Lemma~\ref{lemma:phi-bound-gaussian}, and the lower bound \sloppy{${n! \geq \sqrt{2\pi} n^{n+1/2} \e^{-n}}$} in Stirling's approximation~\citep{Robbins1955}:
  \begin{equation*}
    \begin{split}
    \sup_{x \in \Omega} \, \abs[0]{ f(x) - s_f(x)} &\leq 2 \norm[0]{f}_{\mathcal{H}_K(\Omega)} (1 + n) \frac{ \big(\sqrt{2}\,\varepsilon c_\Omega \big)^n}{\sqrt{n!}} \\
    &\leq 4 \norm[0]{f}_{\mathcal{H}_K(\Omega)} \frac{ \big(\sqrt{2}\,\varepsilon c_\Omega \big)^n n }{\sqrt{n!}} \\
    &\leq \bigg( \frac{128}{\pi} \bigg)^{1/4} \norm[0]{f}_{\mathcal{H}_K(\Omega)} \frac{\big(\sqrt{2}\,\varepsilon c_\Omega \big)^n \e^{n/2} n }{n^{n/2+1/4}} \\
    &= \bigg( \frac{128}{\pi} \bigg)^{1/4} n^{3/4} \norm[0]{f}_{\mathcal{H}_K(\Omega)} \bigg( \frac{\sqrt{2 \e} \, \varepsilon c_\Omega }{n^{1/2}} \bigg)^n \\
    &= C_1 \norm[0]{f}_{\mathcal{H}_K(\Omega)} n^{3/4} \exp\bigg( \! -n \bigg( \frac{1}{2} \log n - \log C_2 \bigg)\bigg).
    \end{split}
  \end{equation*} \qed
\end{proof}

Also Theorem~\ref{thm:superconvergence} can be specialised, and in some cases the kernel of the subspace $\mathcal{H}_K^\alpha(\Omega)$ has an explicit form. 
For instance, set $\alpha_\ell = \sqrt{\ell! \, 2^\ell \varepsilon^{2\ell}}$. Then
\begin{equation*}
  \begin{split}
  K^\alpha(x,y) &= \exp\big(\!-\varepsilon^2(x^2+y^2)\big) \sum_{\ell=0}^\infty \frac{1}{\alpha_\ell^2} \frac{2^\ell \varepsilon^{2\ell}}{\ell!} (xy)^\ell \\
  &= \exp\big(\!-\varepsilon^2(x^2+y^2)\big) \sum_{\ell=0}^\infty \frac{1}{(\ell!)^2} (x y)^\ell,
  \end{split}
\end{equation*}
which can be written in terms of $I_0$, the modified Bessel function of the first kind:
\begin{equation*}
  K^\alpha(x,y) = \exp\big( \!-\varepsilon^2(x^2+y^2)\big) I_0\big(2\sqrt{xy}\big).
\end{equation*}

\begin{theorem} \label{thm:superconvergence-gaussian}
  Consider the Gaussian kernel with the orthonormal expansion~\eqref{eq:gauss-expansion} and suppose that $\Omega \subset \R$ is a closed interval. 
  If $f \in \mathcal{H}_K^\alpha(\Omega)$ is interpolated at the unique approximate Fekete points $\mathcal{X}_n^*$ defined in~\eqref{eq:approximate-fekete-gaussian} and $n \geq 2\varepsilon^2 c_\Omega^2$, then
  \begin{equation*}
    \sup_{x \in \Omega} \, \abs[0]{ f(x) - s_f(x)} \leq C_1 \norm[0]{f}_{\mathcal{H}_K^\alpha(\Omega)} n^{3/4} \alpha_n^{-1} \exp\bigg( \! -n \bigg( \frac{1}{2} \log n - \log C_2 \bigg)\bigg),
  \end{equation*}
  where $C_1 = (128/\pi)^{1/4} \approx 2.53$ and $C_2 = \sqrt{2\e}\, \varepsilon c_\Omega$.
\end{theorem}

If $\Omega = [a,b] \subset \R$ is a closed interval, the standard fill-distance based bound~\citep[Theorem 6.1]{RiegerZwicknagl2010} for interpolation error is
\begin{equation} \label{eq:fill-distance-bound}
  \sup_{ x \in \Omega} \, \abs[0]{ f(x) - s_f(x) } \leq 2 \norm[0]{f}_{\mathcal{H}_K(\Omega)} \exp\big( C \log (h_{\mathcal{X}_n,\Omega}) / h_{\mathcal{X}_n,\Omega} \big)
\end{equation}
whenever the \emph{fill-distance}
\begin{equation*}
  h_{\mathcal{X}_n,\Omega} = \sup_{ x \in \Omega} \, \min_{x_k \in \mathcal{X}_n} \abs[0]{x - x_k}
\end{equation*}
is sufficiently small. The constant in~\eqref{eq:fill-distance-bound} satisfies $C \leq \frac{1}{8}\min\{(b-a)/6, 1\}$.\footnote{This is the constant $C$ in Theorem~6.1 of \citet{RiegerZwicknagl2010}. To derive the claimed bound, observe that this constant is given as $C = \epsilon B/4$ for $B \leq \min\{(b-a)/6, 1\}$ in their proof of Theorem~4.5. On p.\@~120 they show that $\epsilon = 1/2$ if the kernel is Gaussian.} 
For the equispaced points
\begin{equation*}
  %\mathcal{X}_n = \bigg\{ \frac{a}{n}, \frac{2a}{n}, \ldots, \frac{(n-1)a}{n}, a \bigg\},
  \mathcal{X}_n = \bigg\{ a, a + \frac{b-a}{n}, \ldots, b - \frac{b-a}{n}, b \bigg\},
\end{equation*}
which have the minimal fill-distance $h_{\mathcal{X}_n,\Omega} = (b-a)/n$, the bound~\eqref{eq:fill-distance-bound} becomes
\begin{equation*}
\sup_{ x \in \Omega} \, \abs[0]{ f(x) - s_f(x) } \leq 2 \norm[0]{f}_{\mathcal{H}_K(\Omega)} \exp\bigg( - \frac{C}{b-a} \, n \bigg( \log n - \log(b-a) \bigg) \bigg),
\end{equation*}
where $C/(b-a) \leq \frac{1}{48}$.
Our bound~\eqref{eq:error-bound-gaussian} for points $\mathcal{X}_n^*$, being essentially of order $\exp(-\frac{1}{2} n \log n)$, is thus better when $n$ is sufficiently large. However, a significant advantage of bounds of the type~\eqref{eq:fill-distance-bound} is that they apply to nested point sets (i.e., $\mathcal{X}_n \subset \mathcal{X}_{n+1}$ for every $n \geq 1$). It cannot be expected that the approximate Fekete point sets are nested. Further error estimates for Chebyshev-type nodes that cluster near the boundary are provided in~\citep{RiegerZwicknagl2014}.

\begin{remark} It is easy to see that in the Gaussian case the Lagrange basis functions in~\eqref{eq:lagrange-system} can be expressed in terms of the classical polynomial Lagrange functions:
\begin{equation} \label{eq:lagrange-functions}
  u_k^\varphi(x) = \exp( \varepsilon^2 x_k^2) \exp(-\varepsilon^2 x^2) l_k(x),
\end{equation}
where
\begin{equation*}
  l_k(x) = \prod_{\ell \neq k} \frac{x - x_\ell}{x_k - x_\ell}.
\end{equation*}
Let $\Lambda_\text{pol}(\mathcal{X}_n) = \sup_{x \in \Omega} \sum_{k=1}^n \abs[0]{l_k(x)}$ be the Lebesgue constant for polynomial interpolation. It follows easily from~\eqref{eq:lagrange-functions} and the boundedness of $\Omega$ that there exist $C_1, C_2 > 0$ such that
  \begin{equation*}
    C_1 \Lambda_\text{pol}(\mathcal{X}_n) \leq \Lambda_\varphi(\mathcal{X}_n) \leq C_2 \Lambda_\text{pol}(\mathcal{X}_n)
  \end{equation*}
for any $\mathcal{X}_n \subset \Omega$.
This implies that in Theorem~\ref{thm:uniform-error} the coefficient $1+\Lambda_\varphi(\mathcal{X}_n)$ can be replaced with $1+C_2 \Lambda_\text{pol}(\mathcal{X}_n)$, which means that convergence results are available if polynomial Lebesgue constants can be controlled (e.g., if $\mathcal{X}_n$ are the Chebyshev points).
\end{remark}

\subsection{Tensor product algorithms} \label{sec:tensor}

In this section we provide error estimates for interpolation with anisotropic Gaussian kernels in higher dimensions when the evaluation points are constructed as tensor products of the approximate Fekete points~\eqref{eq:approximate-fekete-gaussian}. Besides~\citep{Beatson2010} there does not appear to be much work on error estimates for general anisotropic kernels. \citet{FasshauerHickernell2012} and \citet{SloanWozniakowski2018} analyse the $L^2$-error of general linear algorithms for functions in the RKHS of an anisotropic Gaussian.

Let
\begin{equation} \label{eq:rectangle}
  \Omega = \Omega_1 \times \cdots \times \Omega_d \subset \R^d \quad \text{ for } \quad \Omega_i = [a_i, b_i] \neq \emptyset
\end{equation}
be a hyper-rectangle and consider the $d$-dimensional anisotropic Gaussian kernel~\eqref{eq:gauss-multi},
\begin{equation*}
  K(\b{x}, \b{y}) = \exp \bigg( \! - \sum_{i=1}^d \varepsilon_i^2 (x_i - y_i)^2 \bigg) \eqqcolon \prod_{i=1}^d K_i(x_i, y_i),
\end{equation*}
on $\Omega$.
Let $n_1, \ldots, n_d \in \N$ and denote $N = n_1 \times \cdots \times n_d$. We take the point set to be a tensor product of approximate Fekete point sets~\eqref{eq:approximate-fekete-gaussian} for Gaussian kernels $K_i$ on $\Omega_i$:
\begin{equation} \label{eq:tensor-fekete}
  \mathcal{X}^*_{N} = \mathcal{X}_{1,n_1}^* \times \cdots \times \mathcal{X}_{d,n_d}^* \subset \Omega,
\end{equation}
where $\mathcal{X}_{i,n_i}^* \subset \Omega_i$ stands for the set of $n_i$ approximate Fekete points for kernel $K_i$ on $\Omega_i$. Due to the tensor product structure of the point set and the RKHS~\citep[Section~4.6 in Chapter~1]{BerlinetThomasAgnan2004}, any function $f \in \mathcal{H}_K(\Omega)$ of the form
\begin{equation*}
  f(\b{x}) = f_1(x_1) \times \cdots \times f_d(x_d) \quad \text{ for } \quad  f_1 \in \mathcal{H}_{K_1}(\Omega_1), \ldots, f_d \in \mathcal{H}_{K_d}(\Omega_d)
\end{equation*}
has the norm
\begin{equation*}
  \norm[0]{f}_{\mathcal{H}_K(\Omega)} = \norm[0]{f_1}_{\mathcal{H}_{K_1}(\Omega_1)} \times \cdots \times \norm[0]{f_d}_{\mathcal{H}_{K_d}(\Omega_d)}
\end{equation*}
and the kernel interpolant $s_f$ to any $f \in \mathcal{H}_K(\Omega)$ can be written as
\begin{equation*}
  s_f(\b{x}) = s_{1,f_1}(x_1) \times \cdots \times s_{d,f_d}(x_d),
\end{equation*}
where $s_{i,f_i}$ is the kernel interpolant, based on $K_i$, of $f_i \in \mathcal{H}_{K_i}(\Omega_i)$ at the points $\mathcal{X}_{i,n_i}^*$.

\begin{theorem} \label{thm:error-multi} Consider the multi-dimensional Gaussian kernel~\eqref{eq:gauss-multi} and suppose that $\Omega \subset \R^d$ is a hyper-rectangle of the form~\eqref{eq:rectangle}. If $f \in \mathcal{H}_K(\Omega)$ is interpolated at the tensor product points $\mathcal{X}_N^*$ defined in~\eqref{eq:tensor-fekete} and $n_i \geq 2\varepsilon_i^2 c_{\Omega_i}^2$ for every $i=1,\ldots,d$, then
  \begin{equation} \label{eq:error-bound-multi}
    \sup_{ \b{x} \in \Omega} \, \abs[0]{ f(\b{x}) - s_f(\b{x}) } \leq C_1 \norm[0]{f}_{\mathcal{H}_K(\Omega)} \sum_{i=1}^d n_i^{3/4} \exp\bigg( \! -n_i \bigg( \frac{1}{2} \log n_i - \log C_{i,2} \bigg)\bigg),
  \end{equation}
  where $C_1 = (128/\pi)^{1/4} \approx 2.53$ and $C_{i,2} = \sqrt{2 \e} \, \varepsilon_i c_{\Omega_i}$.
\end{theorem}
\begin{proof}
  Let $g(\b{x}) = g_1(x_1) \times \cdots \times g_d(x_d)$ for $g_i \in \mathcal{H}_{K_i}(\Omega_i)$ and denote $g_{i:j}(\b{x}) = g_i(x_i) \times \cdots \times g_j(x_j)$ for $1 \leq i \leq j \leq d$. Then
  \begin{equation*} %\label{eq:error-decomposition-multi}
    \begin{split}
    g(\b{x}) - s_g(\b{x}) ={}& g_{2:d}(\b{x}) g_1 (x_1) - s_{g_{2:d}}(\b{x}) s_{g_1}(x_1) \\
    ={}& g_{2:d}(\b{x}) [ g_1(x_1) - s_{1,g_1}(x_1)] + [g_{2:d}(\b{x}) - s_{g_{2:d}}(\b{x})] s_{1,g_1}(x_1) \\
    ={}& \sum_{i=1}^{d} [g_i(x_i) - s_{i,g_i}(x_i)] g_{i+1:d}(\b{x}) \prod_{j=1}^{i-1} s_{j,g_j}(x_j),
    \end{split}
  \end{equation*}
  where the notational convention $g_{d+1:d}(\b{x}) = 1$ is used.
  By the reproducing property and the minimum-norm property~\eqref{eq:minimum-norm},
  \begin{equation*}
    \abs[0]{g_i(x_i)} = \abs[1]{ \inprod{g}{K_i(\cdot, x_i)}_{\mathcal{H}_{K_i}(\Omega_i)} } \leq \norm[0]{g_i}_{ \mathcal{H}_{K_i}(\Omega_i)} \: \text{ and } \: \abs[0]{s_{i,g_i}(x_i)} \leq \norm[0]{g_i}_{\mathcal{H}_{K_i}(\Omega_i)}
  \end{equation*}
  for any $i \leq d$ and $x_i \in \Omega_i$. 
  Because $s_{f-s_f} \equiv 0$, the norm $\norm[0]{f}_{\mathcal{H}_K(\Omega)}$ on the right-hand side of the bound~\eqref{eq:error-bound-gaussian} can be replaced with $\norm[0]{f-s_f}_{\mathcal{H}_K(\Omega)}$ by considering interpolation of the function $f-s_f$.
  From this and the above estimates we get
\begin{equation*}
  \begin{split}
    \abs[0]{g(\b{x}) - s_g(\b{x})} &\leq \sum_{i=1}^d \abs[0]{g_i(x_i) - s_{i,g_i}(x_i)} \prod_{j \neq i} \norm[0]{g_j}_{\mathcal{H}_{K_j}(\Omega_j)} \\
    &\leq C_1 \sum_{i=1}^d n_i^{3/4} \exp\bigg( \! -n_i \bigg( \frac{1}{2} \log n_i - \log C_{i,2} \bigg)\bigg) \\
    &\hspace{2cm}\times \norm[0]{g_i - s_{i,g_i}}_{\mathcal{H}_{K_i}(\Omega_i)} \prod_{j \neq i} \norm[0]{g_j}_{\mathcal{H}_{K_j}(\Omega_j)}.
  \end{split}
\end{equation*}
To obtain a bound that is valid for any function in $\mathcal{H}_K(\Omega)$ we exploit~\eqref{eq:P-error-bound}.
For any $\b{x} \in \Omega$ set $g = K(\cdot, \b{x})$.
Because $\norm[0]{K_i(\cdot,x_i)}_{\mathcal{H}_{K_i}(\Omega_i)} = 1$ and, by the power function characterisations~\eqref{eq:P-error-bound} and~\eqref{eq:power-function} and the estimate~\eqref{eq:error-bound-gaussian},
\begin{equation*}
  P_{\mathcal{X}_N^*}(\b{x})^2 = g(\b{x}) - s_g(\b{x})
\end{equation*}
and
\begin{equation*}
  \norm[0]{g_i - s_{i,g_i}}_{\mathcal{H}_{K_i}(\Omega_i)} = P_{\mathcal{X}_{i,n_i}^*}(x_i) \leq C_1 n_i^{3/4} \exp\bigg( \! -n_i \bigg( \frac{1}{2} \log n_i - \log C_{i,2} \bigg)\bigg),
\end{equation*}
we have
\begin{equation*}
  \begin{split}
    P_{\mathcal{X}_N^*}(\b{x})^2 &\leq C_1^2 \sum_{i=1}^d \Bigg[ n_i^{3/4} \exp\bigg( \! - n_i \bigg( \frac{1}{2} \log n_i - \log C_{i,2} \bigg)\bigg) \Bigg]^2 \\
    &\leq \Bigg[ C_1 \sum_{i=1}^d n_i^{3/4} \exp\bigg( \! - n_i \bigg( \frac{1}{2} \log n_i - \log C_{i,2} \bigg)\bigg) \Bigg]^2.
    \end{split}
\end{equation*}
The claim now follows from~\eqref{eq:P-error-bound}. \qed
\end{proof}

In particular, if $n_1 = \cdots = n_d = n$ (so that $N = n^d$) and all $\Omega_i$ and $\varepsilon_i$ are equal, the bound of Theorem~\ref{thm:error-multi} becomes
\begin{equation*}
  \sup_{ \b{x} \in \Omega} \, \abs[0]{ f(\b{x}) - s_f(\b{x}) } \leq C_1 \norm[0]{f}_{\mathcal{H}_K(\Omega)} d \, N^{3/(4d)} \exp\bigg( \! -N^{1/d} \bigg( \frac{1}{2d} \log N - \log C_{2} \bigg)\bigg).
\end{equation*}
It would be straightforward to generalise Theorem~\ref{thm:superconvergence-gaussian} to the tensor product setting.

\section{Numerical examples} \label{sec:examples}

This section contains two numerical examples where maxima of power functions as well as interpolation errors for specific RKHS functions are compared when the interpolation points are either the approximate Fekete, $P$-greedy points, or the Chebyshev points and the kernel is Gaussian.

\begin{figure}[th!]
  \centering
  \includegraphics{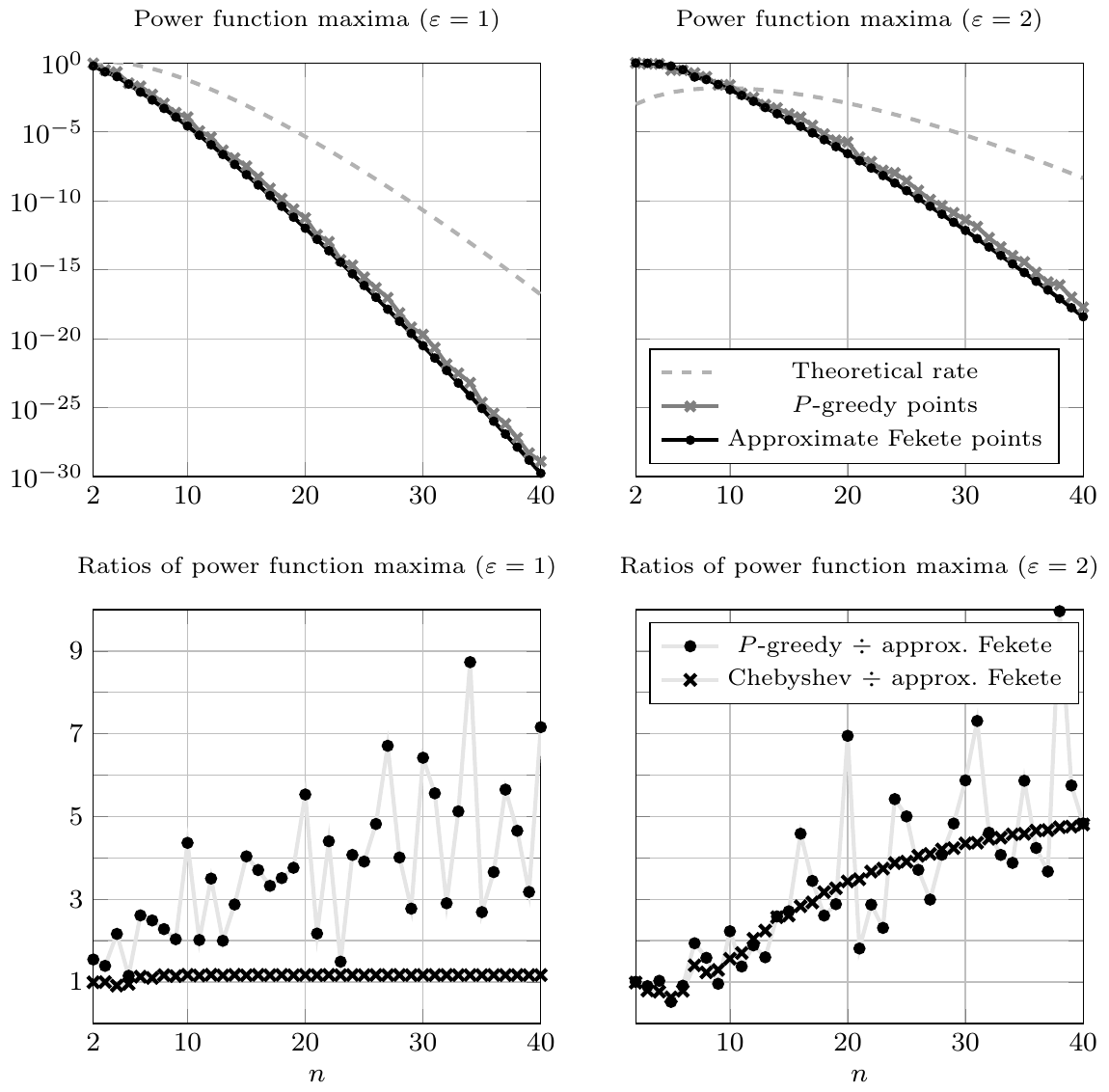}
  \caption{\emph{Top:} Power function maxima $\max_{ x \in [-1, 1] }P_{\mathcal{X}_n}(x)$ approximated using a discretisation of $[-1,1]$ into 1,000 equispaced points for the approximate Fekete and $P$-greedy points of the Gaussian kernel~\eqref{eq:gauss-kernel} with $\varepsilon=1$ (\emph{left}) and $\varepsilon=2$ (\emph{right)}. Also displayed are the (scaled) theoretical rates from Theorem~\ref{thm:error-gaussian}. \emph{Bottom:} Ratios of power function maxima for (1) $P$-greedy and approximate Fekete points and (2) Chebyshev and approximate Fekete points. These panels demonstrate that the power function for the $P$-greedy points can attain a value almost ten times that for the approximate Fekete points ($n=34$ for $\varepsilon=1$ and $n=38$ for $\varepsilon=2$) and that the approximate Fekete points are typically only marginally better than the Chebyshev points when $\varepsilon=1$ but can consistently outperform them when $\varepsilon$ is increased.} \label{fig:power-functions}
\end{figure}

\begin{figure}[th!]
  \centering
  \includegraphics{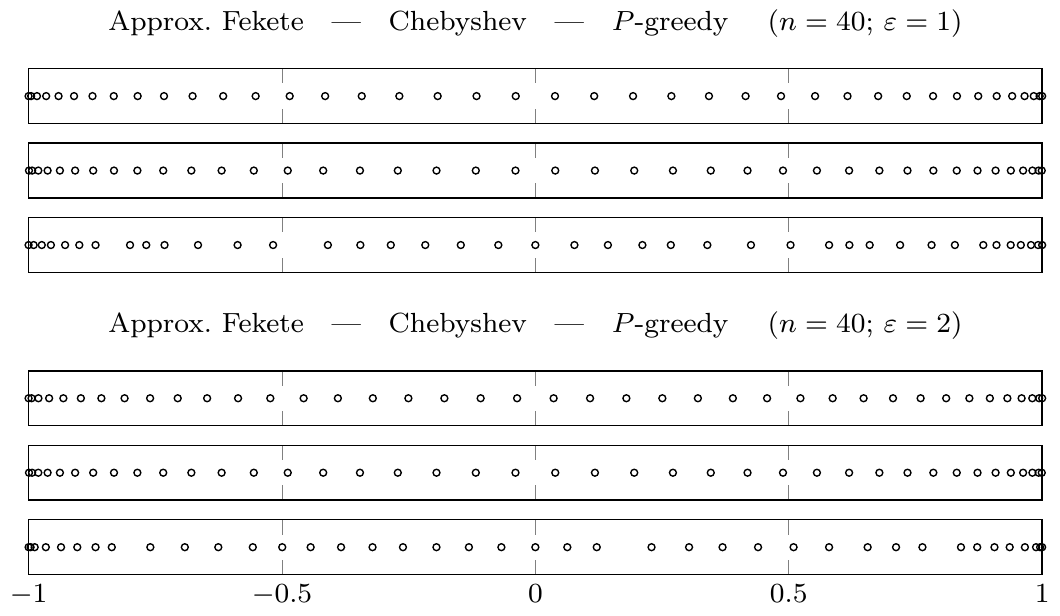}
  \caption{Approximate Fekete points (\emph{top}), Chebyshev points (\emph{middle}), and $P$-greedy points (\emph{bottom}) on $\Omega = [-1,1]$ for the Gaussian kernel~\eqref{eq:gauss-kernel} with $\varepsilon=1$ and $\varepsilon=2$.} \label{fig:points}
\end{figure}

\subsection{Power function} \label{sec:example-power-function}

Figure~\ref{fig:power-functions} displays the maxima of power functions of the univariate Gaussian kernel~\eqref{eq:gauss-kernel} with $\varepsilon = 1$ and $\varepsilon=2$ on $\Omega = [-1,1]$ for three different choices of the interpolation points:
\begin{enumerate}
\item The approximate Fekete points whose construction is outlined in Section~\ref{sec:fekete-gaussian}.
\item The $P$-greedy points, obtained via greedy maximisation of the power function as defined in~\eqref{eq:p-greedy}.
\item The classical Chebyshev points
  \begin{equation*}
    x_k = \cos\bigg( \frac{2k-1}{2n} \pi \bigg) \quad \text{ for } \quad k = 1,\ldots, n
  \end{equation*}
  which do not depend on the choice of the kernel.
\end{enumerate}
The point sets are depicted in Figure~\ref{fig:points} for $n=40$.
The $P$-greedy points as well as the power function maxima were computed by discretising the interval into 1,000 equispaced points.
That is, the next $P$-greedy point was always solved from
  \begin{equation} \label{eq:P-greedy-discrete}
    x_{n+1} \in \argmax_{ x \in \Omega_h } P_{\mathcal{X}_n}(x),
  \end{equation}
  where $\Omega_h = \{-1,-1+h,\cdots,1-h,1\}$ and $h=\frac{1}{999}$, is a uniform discretisation of $[-1,1]$.
The results show that the approximate Fekete points outperform the $P$-greedy points and the Chebyshev points.
Given Remark~\ref{rmk:flat-limit} it is not surprising that the approximate Fekete points are only marginally better than the Fekete points when the relatively small value $\varepsilon=1$ is used.
We also see that the approximate Fekete points are very close, but not identical, to the Chebyshev points when $\varepsilon=1$ and that they cover the domain more uniformly than the $P$-greedy points for the both values of $\varepsilon$ used.

As proved in Proposition~\ref{prop:convex}, the approximate Fekete points are solved from a convex optimisation problem.
Computing the next $P$-greedy point in~\eqref{eq:P-greedy-discrete} requires finding the maximum of $P_{\mathcal{X}_n}$ on the finite set $\Omega_h$, and $P_{\mathcal{X}_n}$ can be updated to step $n+1$ on $\Omega_h$ at a computational cost of $\mathcal{O}(n^2 \abs[0]{\Omega_h})$.
On the downside, it should be noted that the power function quickly becomes numerically unstable due to severe ill-conditioning of the kernel matrix of the Gaussian kernel.
The superiority of the the approximate Fekete points from computational perspective is demonstrated by our implementation which used MATLAB's native \texttt{fmincon} function to efficiently compute the approximate Fekete points without domain discretisation but had to resort to costly arbitrary-precision arithmetic (\textsf{mpmath} library~\citep{mpmath2013} in Python) for numerically stable computation of the $P$-greedy points (arbitrary-precision arithmetic was also used to compute the power function maxima for all point sets).
This makes a straightforward comparison of computational complexities of the two methods difficult.

\subsection{Specific RKHS functions} \label{sec:example-function}

\begin{figure}[th!]
  \centering
  \includegraphics{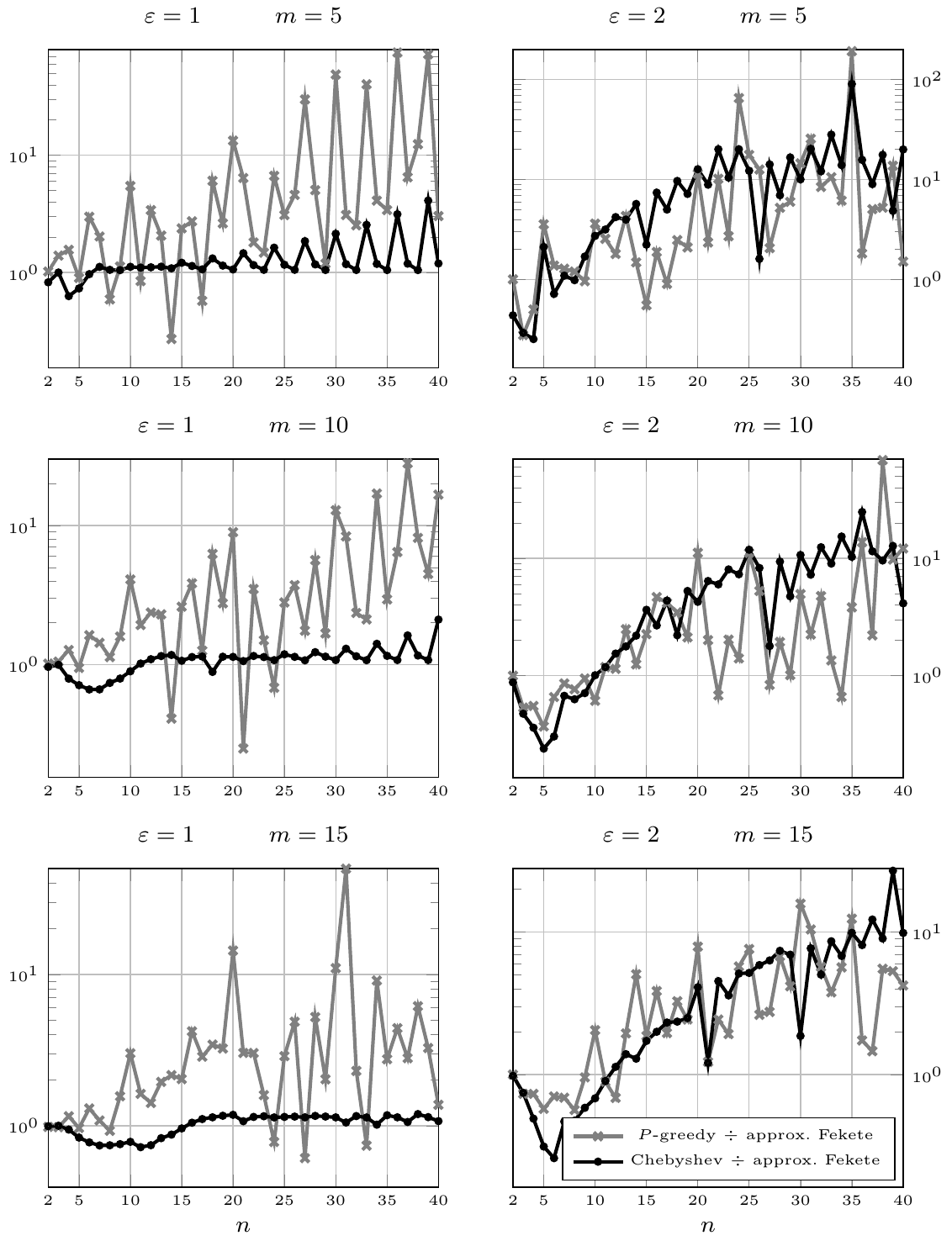}
  \caption{The ratios~\eqref{eq:ratios} of maximal errors in interpolating the function~\eqref{eq:test-function} using the kernel interpolant~\eqref{eq:kernel-interpolant} based on the Gaussian kernel~\eqref{eq:gauss-kernel} with $\varepsilon = 1$ (\emph{left}) and $\varepsilon=2$ (\emph{right}). Ratios larger than one mean that the approximate Fekete points outperform the $P$-greedy points or Chebyshev points in terms of the selected error criterion.} \label{fig:rkhs-functions}
\end{figure}

We use the kernel interpolant~\eqref{eq:kernel-interpolant} based on the Gaussian kernel~\eqref{eq:gauss-kernel} with $\varepsilon \in \{1,2\}$ to approximate the functions
\begin{equation} \label{eq:test-function}
  f_{\varepsilon,m}(x) = x^m \exp( x - \varepsilon^2 x^2 )
\end{equation}
for $m \in \{5,10,15\}$ on $\Omega = [-1,1]$.
Using~\eqref{eq:gauss-rkhs} and the expansion
\begin{equation*}
  f_{\varepsilon,m}(x) = x^m \sum_{\ell=0}^\infty \frac{1}{\ell!} x^\ell \exp( -\varepsilon^2 x^2 )
\end{equation*}
we compute that
\begin{equation*}
    \norm[0]{f_{\varepsilon,m}}_{\mathcal{H}_K(\Omega)}^2 = \sum_{\ell=m}^\infty \frac{\ell!}{2^{\ell} \varepsilon^{2\ell} ((\ell-m!))^2} = (2 \varepsilon^2)^{-m} \sum_{\ell=0}^\infty (2 \varepsilon^2)^{-\ell} \frac{(\ell+m)!}{(\ell!)^2},
\end{equation*}
which can be proved to converge by using, for example, the ratio test.
This verifies that $f_{\varepsilon,m} \in \mathcal{H}_K(\Omega)$ for every $m \in \N$ and $\varepsilon > 0$.

The results are displayed in Figure~\ref{fig:rkhs-functions} in terms of the ratios of maximal interpolation errors,
\begin{equation} \label{eq:ratios}
  \frac{\sup_{ x \in [-1,1] }\abs[0]{f_{\varepsilon,m}(x) - s_{f_{\varepsilon,m}}(x)}}{\sup_{ x \in [-1,1] } \abs[0]{f_{\varepsilon,m}(x) - s_{f_{\varepsilon,m}}^*(x)}},
\end{equation}
where the interpolant in the numerator uses either the $P$-greedy points or the Chebyshev points and the interpolant in the denominator uses the approximate Fekete points.
As in Section~\ref{sec:example-power-function}, the suprema were approximated using the 1,000-point equispaced discretisation of the interval and arbitrary-precision arithmetic.
The results show that the approximate Fekete points fairly consistently outperform the two alternatives, particularly when the number of points and the scale parameter are large ($n \geq 15$ and $\varepsilon=2$).
The results for $\varepsilon=1$ closely mirror those for the power function in Section~\ref{sec:example-power-function} in that the improvement over the Chebyshev points is only marginal.

\section*{Acknowledgements}

T.\@~Karvonen was supported by the Aalto ELEC Doctoral School and the Lloyd's Register Foundation programme on data-centric engineering at the Alan Turing Institute, United Kingdom. This research was partially carried out while he was visiting the University of Tokyo, funded by the Finnish Foundation for Technology Promotion and Oskar Öflunds Stiftelse. S.\@~Särkkä was supported by the Academy of Finland.
K.\@~Tanaka was supported by the grant-in-aid of Japan Society of the Promotion of Science with KAKENHI Grant Number~17K14241.

%\bibliography{references}

\end{document}